\documentclass[11pt]{amsart}
\usepackage{amsmath}
\usepackage{amssymb}
\usepackage{amsfonts}
\usepackage{mathrsfs}

\numberwithin{equation}{section}       % Number formulas within sections
\numberwithin{figure}{section}       % Number figures within sections
\theoremstyle{plain}

\newtheorem{prop}{Proposition}[section]

\newtheorem{coro}[prop]{Corollary}
\newtheorem{lemm}[prop]{Lemma}
\newtheorem{fact}[prop]{Fact}
\newtheorem{theoalph}{Theorem}

\theoremstyle{definition}
\newtheorem{defi}[prop]{Definition}

\theoremstyle{remark}
\newtheorem{rema}[prop]{Remark}

\newtheorem{prob}[prop]{Problem}

\newtheoremstyle{citing}% name
  {3pt}%      Space above, empty = `usual value'
  {3pt}%      Space below
  {\itshape}% Body font
  {}%         Indent amount (empty = no indent, \parindent = para indent)
  {\bfseries}% Thm head font
  {.}%        Punctuation after thm head
  {.5em}%     Space after thm head: " " = normal interword space;
  {\thmnote{#3}}% Thm head spec

\theoremstyle{citing}
\newtheorem*{generic}{}% all text supplied in the note

\newcommand{\partn}[1]{{\smallskip \noindent \textbf{#1.}}}% for parts of a proof

\newcommand{\C}{\mathbb{C}}

\newcommand{\R}{\mathbb{R}}

\newcommand{\Z}{\mathbb{Z}}

\newcommand{\cM}{\mathcal{M}}

\newcommand{\cZ}{\mathcal{Z}}

\newcommand{\sA}{\mathscr{A}}
\newcommand{\sB}{\mathscr{B}}

\newcommand{\hC}{\widehat{C}}

\newcommand{\hK}{\widehat{K}}

\newcommand{\hS}{\widehat{S}}

\newcommand{\hvarphi}{\widehat{\varphi}}

\newcommand{\tB}{\widetilde{B}}

\newcommand{\tM}{\widetilde{M}}

\newcommand{\tU}{\widetilde{U}}

\newcommand{\tW}{\widetilde{W}}

\newcommand{\teta}{\widetilde{\teta}}

\newcommand{\tnu}{\widetilde{\nu}}

\newcommand{\tphi}{\widetilde{\phi}}
\newcommand{\tvarphi}{\widetilde{\varphi}}

\newcommand{\tpsi}{\widetilde{\psi}}

\renewcommand{\=}{ : = }

\DeclareMathOperator{\diam}{diam}
\DeclareMathOperator{\Lip}{Lip} % best Lipschitz constant norm notation
\DeclareMathOperator{\dist}{dist}

\DeclareMathOperator{\Crit}{Crit} % Set of critical points
\DeclareMathOperator{\CV}{CV} % Set of critical values
\newcommand{\CC}{\overline{\C}}% Riemann sphere
\newcommand{\htop}{h_{\operatorname{top}}}

\DeclareMathOperator{\dom}{dom}

\newcommand{\eps}{\varepsilon}
\newcommand{\ul}{\underline{l}}

\newcommand{\whf}{\widehat{f}}

\begin{document}

\title[Equilibrium states of one-dimensional maps]{Equilibrium states of weakly hyperbolic one-dimensional maps for H{\"o}lder potentials}

\author{Huaibin Li}
\address{Huaibin Li, Facultad de Matem{\'a}ticas, Pontificia Universidad Cat{\'o}lica de Chile, Avenida Vicu{\~n}a Mackenna~4860, Santiago, Chile}
\email{matlihb@gmail.com}
\thanks{HL was partially supported by the National Natural Science Foundation of China (Grant No. 11101124) and FONDECYT grant 3110060 of Chile.}

\author{Juan Rivera-Letelier}
\address{Juan Rivera-Letelier, Facultad de Matem{\'a}ticas, Pontificia Universidad Cat{\'o}lica de Chile, Avenida Vicu{\~n}a Mackenna~4860, Santiago, Chile}
\email{riveraletelier@mat.puc.cl}
\thanks{JRL was partially supported by FONDECYT grant 1100922 of Chile.}


\begin{abstract}
There is a wealth of results in the literature on the thermodynamic formalism for potentials that are, in some sense, ``hyperbolic''.
We show that for a sufficiently regular one-dimensional map satisfying a weak hyperbolicity assumption, every H{\"o}lder continuous potential is hyperbolic.
A sample consequence is the absence of phase transitions: The pressure function is real analytic on the space of H{\"o}lder continuous functions.
Another consequence is that every H{\"o}lder continuous potential has a unique equilibrium state, and that this measure has exponential decay of correlations.
\end{abstract}

\maketitle

\section{Introduction}
The thermodynamic formalism of smooth dynamical systems was initiated  by Sinai, Ruelle, and Bowen~\cite{Bow75,Rue76,Sin72}.
For a uniformly hyperbolic diffeomorphism acting on a compact manifold of arbitrary dimension, they gave a complete description for H{\"o}lder continuous potentials.
In particular, they showed that on each basic set there is a unique equilibrium state, and that this measure has exponential decay of correlations.
Furthermore, there are no phase transitions: The topological pressure of a basic set is real analytic as a function of the potential.

There have been several extensions of these results to one-dimensional maps, that go beyond the uniformly hyperbolic setting.
The lack of uniform hyperbolicity is usually compensated by an extra hypothesis on the potential.
For example, there is a wealth of results for a piecewise monotone interval map~$f : I \to I$ and a potential~$\varphi$ of bounded variation satisfying
$$ \sup_I \varphi < P(f, \varphi), $$
where~$P(f, \varphi)$ denotes the pressure.
Most results apply under the following weaker condition:
\begin{center}
For some integer~$n \ge 1$, the function~$S_n(\varphi) \= \sum_{j = 0}^{n - 1} \varphi \circ f^j$ satisfies
$\sup_{I} \frac{1}{n} S_n(\varphi) < P(f, \varphi).$
\end{center}
In what follows, a potential~$\varphi$ satisfying this condition is said to be \emph{hyperbolic for~$f$}.
See for example~\cite{BalKel90,DenKelUrb90,HofKel82b,Kel85,LivSauVai98,Rue94a} and references therein, as well as Baladi's book~\cite[\S$3$]{Bal00b}.
The classical result of Lasota and Yorke~\cite{LasYor73} can be recovered as the special case in which~$f$ is piecewise~$C^2$ and uniformly expanding, and~$\varphi = - \log |Df|$.

For a complex rational map in one variable~$f$, and a H{\"o}lder continuous potential~$\varphi$ that is hyperbolic for~$f$, a complete description of the thermodynamic formalism was given by Denker, Haydn, Przytycki, and Urba{\'n}ski in~\cite{Hay99,DenPrzUrb96,DenUrb91e,Prz90},\footnote{In this setting most of the results have been stated for a potential~$\varphi$ satisfying the condition $\sup \varphi < P(f, \varphi)$ that is more restrictive than~$\varphi$ being hyperbolic for~$f$. However, general arguments show they also apply to hyperbolic potentials, see~\cite[\S$3$]{InoRiv12}.} extending previous results of Freire, Lopes, and Ma{\~n}{\'e}~\cite{FreLopMan83,Man83}, and Ljubich~\cite{Lju83}.
See also the alternative approach of Szostakiewicz, Urba{\'n}ski, and Zdunik in~\cite{SzoUrbZdu11one}, and~\cite{FisUrb00,Hay00,InoRiv12} for further results.

Using the results of Keller in~\cite{Kel85}, we recently extended these results to the case of a sufficiently regular interval map and a H{\"o}lder continuous potential in the companion paper~\cite{LiRiv1210a}.
See~\cite{BruTod08} for an earlier result in a similar setting.

Suitably interpreted, the hyperbolicity condition also appears in the study of geometric potentials in~\cite{PrzRiv11,PrzRivinterval}.
In fact, consider a continuous map~$f$ acting on a compact metric space~$X$, and let~$\cM(X, f)$ be the space of Borel probability measures on~$X$ that are invariant by~$f$.
Then a continuous potential~$\varphi$ is hyperbolic if and only if
\begin{equation}
  \label{e:canonic hyperbolicity}
  \sup_{\mu \in \cM(X, f)} \int_X \varphi d \mu < P(f, \varphi),
\end{equation}
see for example~\cite[Proposition~$3.1$]{InoRiv12}.
In the case~$f$ is a differentiable map in one real or complex variable, put
$$ \chi_{\inf} \= \inf_{\mu \in \cM(X, f)} \int_X \log |Df| \ d \mu
\text{ and }
\chi_{\sup} \= \sup_{\mu \in \cM(X, f)} \int_X \log |Df| \ d \mu. $$
Then for each real number~$t$, the condition~\eqref{e:canonic hyperbolicity} applied to the potential~$\varphi = - t \log |Df|$ becomes
$$ P(f, - t \log |Df|)
>
\max \left\{ - t \chi_{\inf}, - t \chi_{\sup} \right\}, $$
which is precisely the condition appearing in~\cite{PrzRiv11,PrzRivinterval}.

In this paper we show, in both the real and the complex setting, that for a sufficiently regular one-dimensional map satisfying a weak form of hyperbolicity, \textit{\textsf{every}} H{\"o}lder continuous potential is hyperbolic.
So, for such a map~$f$, all the results above apply to every H{\"o}lder continuous potential.
A consequence is that for every H{\"o}lder continuous function~$\varphi$ there is a unique equilibrium state of~$f$ for the potential~$\varphi$.
Moreover, this measure has several statistical properties, like exponential decay of correlations and the Almost Sure Invariance Principle, see Remark~\ref{r:statistical properties}.
Another consequence is the absence of phase transitions: The pressure function is real analytic on the space of H{\"o}lder continuous functions.
In~\cite{Li1307} this result is used to show a level-$2$ Large Deviation Principle for H{\"o}lder continuous potentials.

We proceed to describe our main results more precisely.

\subsection{Statements of results}
\label{ss:statements}
We state~$2$ results, one in the real setting and the other in the complex setting.
To simplify the exposition, they are formulated in a more restricted situation than what we are able to handle, see the Main Theorem in~\S\ref{s:a reduction} for a more general formulation of our results.
We start recalling some concepts from thermodynamic formalism, see for example~\cite{PrzUrb10,Wal78} for background.

Let~$(X, \dist)$ be a compact metric space and $T:X\to X$ a continuous map.
Denote by~$\cM(X)$ the space of Borel probability measures on~$X$ endowed with the weak* topology, and by~$\cM(X, T)$ the subspace of~$\cM(X)$ of those measures that are invariant by~$T$.
For each measure~$\nu$ in~$\cM(X, T)$, denote by~$h_{\nu}(T)$ the \emph{measure-theoretic entropy} of~$\nu$.
For a continuous function $\varphi: X \to \R$, denote by~$P(T,\varphi)$ the \emph{topological pressure of $T$ for the potential $\varphi$}, defined by
\begin{equation}
\label{e:variational principle}
P(T, \varphi)
\=
\sup\left\{h_\nu(T) + \int_X \varphi \ d\nu: \nu\in \cM(X, T)\right\}.
\end{equation}
A measure~$\nu$ in~$\cM(X, T)$ is called an \emph{equilibrium state of $T$ for the potential~$\varphi$}, if the supremum in~\eqref{e:variational principle} is attained at~$\nu$.

On the other hand, a measure~$\nu$ in~$\cM(X, T)$ is \emph{exponentially mixing for~$f$} or \emph{has exponential decay of correlations for~$f$}, if there are constants~$C > 0$ and~$\rho$ in~$(0, 1)$, such that for every bounded measurable function~$\phi : X \to \R$ and every Lipschitz continuous function $\psi : X \to \R$, we have for every integer~$n \ge 1$
$$ \left|\int_{X}\phi \circ f^n \cdot \psi \ d\nu - \int_{X}\phi \ d\nu \int_{X} \psi \ d\nu \right|
\le
C \left( \sup_X |\phi| \right) \| \psi \|_{\Lip} \cdot \rho^n, $$
where $\| \psi \|_{\Lip} \= \sup_{x, x' \in X, x \neq x'} \frac{|\psi(x) - \psi(x')|}{\dist(x,x')}$.

Given a compact interval~$I$, a smooth map~$f : I \to I$ is \emph{non-degenerate}, if the number of points at which the derivative of~$f$ vanishes is finite, and if for every such point there is a higher order derivative of~$f$ that is non-zero. 
\begin{theoalph}
\label{t:real}
Let~$I$ be a compact interval and let $f: I \to I$ be a topologically exact non-degenerate smooth map.
Assume~$f$ has only hyperbolic repelling periodic points, and for each critical value~$v$ of~$f$ we have
$$ \lim_{n\to + \infty} |Df^n(v)| = + \infty. $$
Then every H\"{o}lder continuous potential $\varphi: I \rightarrow \R$ is hyperbolic for~$f$.
In particular, there is a unique equilibrium state~$\nu$ of~$f$ for the potential~$\varphi$.
Moreover, the measure-theoretic entropy of~$\nu$ is positive, and~$\nu$ is exponentially mixing for~$f$.
Finally, for every H{\"o}lder function~$\psi : I \to \R$, the function~$t \mapsto P(f, \varphi + t \psi)$ is real analytic.
\end{theoalph}

Under the additional assumption that~$\varphi$ is hyperbolic for~$f$, the assertions of Theorem~\ref{t:real} can be obtained by combining~\cite[Corollaries~$1.1$ and~$1.2$]{BuzPacSch01} and~\cite[Theorem~A]{RivShe1004}.
Under the more restrictive \emph{bounded range condition}: $\sup_I\varphi - \inf_I \varphi < \htop(f)$, these are given by~\cite[Theorem~$3.4$]{Kel85}, see also~\cite{BruTod08, BruTod11}.
Theorem~\ref{t:real} shows that the assumption that~$\varphi$ is hyperbolic for~$f$ is not needed.

Buzzi showed a result related to Theorem~\ref{t:real} for potentials that are ``symbolically H{\"o}lder continuous'', see~\cite[Theorem~$1$]{Buz04}.
For a smooth non-degenerate map~$f$, the class of symbolically H{\"o}lder continuous potentials is not related to the class of H{\"o}lder continuous potentials, unless~$f$ satisfies the Topological Collet-Eckmann condition, see~\cite[Corollary~C]{Riv1204}.
If~$f$ satisfies the Topological Collet-Eckmann condition, then these classes of potentials coincide and our argument gives a different proof of Buzzi's result.\footnote{Although there are maps satisfying the Topological Collet-Eckmann condition that do not satisfy the hypotheses of Theorem~\ref{t:real}, our Main Theorem in~\S\ref{s:a reduction} does apply to maps satisfying the Topological Collet-Eckmann condition, see Remark~\ref{r:TCE}.}
When~$f$ has only one critical value~$v$, the Topological Collet-Eckmann condition is equivalent to the condition that~$|D f^n(v)|$ grows exponentially in~$n$, so our result is significantly stronger in this case.

Once it is shown that the potential~$\varphi$ is hyperbolic for~$f$, the rest of the assertions of Theorem~\ref{t:real} follow from~\cite[Corollary~$1.4$]{LiRiv1210a}.

In contrast with Theorem~\ref{t:real}, our Main Theorem in~\S\ref{s:a reduction} applies to maps having a neutral periodic point.
The prototypical examples of interval maps having a neutral periodic point are the intermittent maps studied in~\cite{LivSauVai99,Lop93,MelTer12,PomMan80,PreSla92,PolShaYur98,PolWei99,Sar01a,You99}, among others.
For concreteness, fix~$\alpha$ in~$(0, 1)$ and consider the map
\begin{center}
  \begin{tabular}{rcl}
    $f_{\alpha} : [0, 1]$ & $\to$ & $[0, 1]$ \\
$x$ & $\mapsto$ &
$\begin{cases}
x(1 + x^{\alpha}) & \text{if~$x(1 + x^\alpha) \le 1$;} \\
x(1 + x^\alpha) - 1 & \text{otherwise.}
\end{cases}$
  \end{tabular}
\end{center}
Our argument shows that for each~$\alpha'$ in~$(\alpha, 1]$, every H{\"o}lder continuous potential of exponent~$\alpha'$ is hyperbolic for~$f_\alpha$.
In particular, $f_\alpha$ has no phase transitions on the space of H{\"o}lder continuous potentials of exponent~$\alpha'$.
The hypothesis that~$\alpha'$ is in~$(\alpha, 1]$ is necessary: The potential~$- \log |Df_\alpha|$ is H{\"o}lder continuous of exponent~$\alpha$, and it is not hyperbolic for~$f_\alpha$, see Remark~\ref{r:intermittent}.

To state our main result in the complex setting, for each complex rational map~$f$ denote by~$\Crit(f)$ the set of critical points of~$f$, and by~$J(f)$ the Julia set of~$f$.
In what follows we restrict the action of~$f$ to~$J(f)$.
In particular, the pressure function is defined through measures supported on~$J(f)$ and each equilibrium state is supported on~$J(f)$.
\begin{theoalph}
\label{t:complex}
Let~$f$ be one of the following:
\begin{enumerate}
\item[1.]
An at most finitely renormalizable complex polynomial of degree at least~$2$, without neutral cycles and such that for each critical value~$v$ of~$f$ we have
$$ \lim_{n\rightarrow + \infty}|Df^n(v)| = + \infty; $$
\item[2.]
A complex rational map of degree at least~$2$, without parabolic cycles and such that for every critical value~$v$ of~$f$ we have
$$ \sum_{n = 1}^{+ \infty} \frac{1}{|Df^n(v)|} < + \infty. $$
\end{enumerate}
Then every H\"{o}lder continuous potential $\varphi: J(f)\rightarrow \R$ is hyperbolic for~$f$.
In particular, there is a unique equilibrium state~$\nu$ of~$f$ for the potential~$\varphi$.
Moreover, the measure-theoretic entropy of~$\nu$ is positive and~$\nu$ is exponentially mixing for~$f$.
Finally, for every H{\"o}lder function~$\psi : I \to \R$, the function~$t \mapsto P(f, \varphi + t \psi)$ is real analytic.
\end{theoalph}
Recall that for an integer~$s \ge 2$, a complex polynomial~$f$ is \emph{renormalizable of period~$s$}, if there are Jordan disks~$U$ and~$V$ in~$\C$, such that~$\overline{U} \subset V$ and such that the following hold:
\begin{itemize}
\item $f^s: U\to V$ is proper of degree at least~$2$;
\item The set $\{z\in U: f^{sn}(z)\in U \text{ for all } n=1,2,\ldots\}$ is connected and strictly contained in~$J(f)$;
\item For each critical point~$c$ of~$f$ there exists at most one~$j$ in~$\{0,1,\ldots, s-1\}$ such that~$c$ is~$f^j(U)$.
\end{itemize}
We say that~$f$ is \emph{infinitely renormalizable} if there are infinitely many integers~$s \ge 2$ such that~$f$ is renormalizable of period~$s$.

In the case where~$f$ is a complex rational map satisfying the Topological Collet-Eckmann condition, the essential part of Theorem~\ref{t:complex} is~\cite[Corollary~$1.2$]{InoRiv12}.
We note that the results of~\cite{InoRiv12} only apply to maps having no invariant measure of zero Lyapunov exponent, so Theorem~\ref{t:complex} is substantially stronger.

Once it is shown that the potential~$\varphi$ is hyperbolic for~$f$, the rest of the assertions of Theorem~\ref{t:complex} follow from either~\cite{DenPrzUrb96,Hay99} or~\cite{SzoUrbZdu11one}, combined with general arguments in~\cite[\S$3$]{InoRiv12}.

\begin{rema}
\label{r:statistical properties}
The equilibrium states in Theorem~\ref{t:real} are found in~\cite{LiRiv1210a} using a result of Keller in~\cite{Kel85}, showing that the corresponding transfer operator has a spectral gap in a certain Banach space containing the space of Lipschitz functions.
On the other hand, the equilibrium states in Theorem~\ref{t:complex} are found in~\cite{SzoUrbZdu11one} through a Young tower with an exponential tail estimate.
Thus, in both settings the equilibrium states satisfy several statistical properties, such as the Central Limit Theorem and the Almost Sure Invariance Principle, see~\cite[Theorem~$3.3$]{Kel85} for the former and, \emph{e.g.}, \cite{Gou05,MelNic05,MelNic08,Tyr05} for the latter.
\end{rema}

A natural question raised by Theorems~\ref{t:real} and~\ref{t:complex} is what kind of phase transitions can occur for a smooth one dimensional map, on the space of H{\"o}lder continuous potentials.
It is well-known that neutral periodic points are responsible for phase transitions on the space of H{\"o}lder continuous potentials of a sufficiently small exponent, see for example~\cite[Theorem~$3$]{Lop93} or~\cite[Proposition~$1$]{Sar01a}.
A phase transition for a map without neutral periodic points would be more subtle to produce.
``Fibonacci maps'' are a natural candidate to have a phase transition.
Given~$\ell > 1$ and a real parameter~$c$, an interval map of the form~$x \mapsto |x|^{\ell} + c$ is a \emph{Fibonacci map} if, roughly speaking, the closest return times of the orbit of the critical point~$x = 0$ occur at the Fibonacci numbers, see for example~\cite[Theorem~$1.1$]{LyuMil93} for a precise definition in terms of kneading sequences.
General results imply that for each real number~$\ell > 1$, there is a parameter~$c(\ell)$ such that a suitable restriction~$f_\ell$ of~$x \mapsto |x|^\ell + c(\ell)$ is a Fibonacci map.
For such~$f_\ell$, the closure~$\omega(\ell)$ of the critical orbit is a Cantor set, see for example~\cite[Theorem~$1.2$]{LyuMil93}.
The results of~\cite{KelNow95} imply that there is~$\ell_1 > 2$ such that for~$\ell$ in~$(0, \ell_1)$ the map~$f_\ell$ satisfies the hypothesis of Theorem~\ref{t:real}.
So for~$\ell$ in~$(0, \ell_1)$, the map~$f_\ell$ does not have phase transitions in the space of H{\"o}lder continuous potentials.

\begin{prob}
Show that for every sufficiently large~$\ell > 1$, there are constants~$C > 0$ and~$\alpha$ in~$(0, 1)$ such that the potential~$\varphi_0(x) \= - C \dist(x, \omega(\ell))^\alpha$ satisfies~$P(f_\ell, \varphi_0) = 0$.
\end{prob}

A positive solution of this problem would imply that~$\varphi_0$ is not hyperbolic for~$f_\ell$, and that the function~$t \mapsto P(f_\ell, t \varphi_0)$ is not real analytic.
\subsection{Organization and strategy of the proof}
\label{ss:organization}
To prove Theorems~\ref{t:real} and~\ref{t:complex}, we follow the general strategy of the proof of~\cite[Main Theorem]{InoRiv12}.
The main step in the proof of~\cite[Main Theorem]{InoRiv12} is the construction, for a given invariant measure with positive Lyapunov exponent, of an Iterated Function System (IFS) that is used to estimate the pressure from below.
In this paper, we have to deal with the more difficult situation of an invariant measure of zero Lyapunov exponent.
For such a measure, the construction of a suitable IFS does not seem possible.
Instead, we consider a more general type of induced system, formed by multivalued functions.
Once this induced system is constructed, we follow the strategy in~\cite{InoRiv12} to deduce our main results, although in the real case there are some extra difficulties coming from the fact that interval maps are not open maps in general.

The paper is organized as follows.
In~\S\ref{s:a reduction} we state a version of Theorems~\ref{t:real} and~\ref{t:complex} that holds for a more general class of maps; it is stated as ``Main Theorem''.
After deriving Theorems~\ref{t:real} and~\ref{t:complex} from the Main Theorem and known results in~\S\ref{ss:proof of real and complex}, we state our main technical result as the ``Key Lemma'' in~\S\ref{ss:a reduction}, where we also derive the Main Theorem from it.

In~\S\ref{s:constructing a free IMFS} we construct an induced system formed by multivalued functions, and in~\S\ref{s:proof of Key Lemma} we prove the Key Lemma using this induced system to estimate the pressure from below.

\subsection{Acknowledgments}
\label{ss:Acknowledgements}
The authors would like to thank Henk Bruin and Godofredo Iommi for useful comments on the first version.

\section{A reduction}
\label{s:a reduction}
We start this section stating a version of Theorems~\ref{t:real} and~\ref{t:complex} that holds for a more general class of maps; it is stated as the ``Main Theorem''.
In~\S\ref{ss:proof of real and complex} we derive Theorems~\ref{t:real} and~\ref{t:complex} as a direct consequence of the Main Theorem and known results.
In~\S\ref{ss:a reduction} we state our main technical result as the ``Key Lemma'', and we derive the Main Theorem from it.

To state the Main Theorem, we introduce a class of interval maps that includes non-degenerate smooth maps as special cases.
Let~$I$ be a compact interval.
For a differentiable map~$f : I \to I$, a point of~$I$ is \emph{critical} if the derivative of~$f$ vanishes at it.
Denote by~$\Crit(f)$ the set of critical points of~$f$.
A differentiable interval map~$f : I \to I$ is \emph{of class~$C^3$ with non-flat critical points}, if it has a finite number of critical points and if:
\begin{itemize}
\item
The map~$f$ is of class~$C^3$ outside $\Crit(f)$;
\item
For each critical point~$c$ of~$f$ there exists a number $\ell_c>1$ and diffeomorphisms~$\phi$ and~$\psi$ of~$\R$ of class~$C^3$, such that $\phi(c)=\psi(f(c))=0$, and such that on a neighborhood of~$c$ on~$I$, we have
$$ |\psi\circ f| = |\phi|^{\ell_c}. $$
\end{itemize}
Note that each smooth non-degenerate interval map is of class~$C^3$ with non-flat critical points, and that each map of class~$C^3$ with non-flat critical points is continuously differentiable.

For an interval map of class~$C^3$ with non-flat critical points~$f$, denote by~$\dom(f)$ the interval on which~$f$ is defined, and denote by~$\dist$ the distance on~$\dom(f)$ induced by the norm distance on~$\R$.
On the other hand, for a complex rational map~$f$ we use~$\dom(f)$ to denote the Riemann sphere~$\CC$, which we endow with the spherical metric, that we also denote by~$\dist$.
In both, the real and complex setting, for a subset~$W$ of~$\dom(f)$ we use~$\diam(W)$ to denote the diameter of~$W$ with respect to~$\dist$.

\begin{defi}
Let~$f$ be either a complex rational map, or an interval map of class~$C^3$ with non-flat critical points.
The \emph{Julia set~$J(f)$ of~$f$} is the complement of the largest open subset of~$\dom(f)$ on which the family of iterates of~$f$ is normal.
\end{defi}

If~$f$ is a complex rational map of degree at least~$2$, then~$J(f)$ is a perfect set that is equal to the closure of repelling periodic points.
Moreover, $J(f)$ is completely invariant and~$f$ is topologically exact on~$J(f)$.
We denote by~$\sA_{\C}$ the collection of all rational maps of degree at least~$2$.

In contrast with the complex setting, the Julia set of an interval map of class~$C^3$ with non-flat critical points might be empty, reduced to a single point, or might not be completely invariant.\footnote{This last property can only happen if there is a turning point in the interior of the basin of a one-sided attracting neutral periodic point, that is eventually mapped to this neutral periodic point.}
However, if the Julia set of such a map~$f$ is not completely invariant, then it is possible to make an arbitrarily small smooth perturbation of~$f$ outside a neighborhood of~$J(f)$, so that the Julia set of the perturbed map is completely invariant and coincides with~$J(f)$. 
We denote by~$\sA_{\R}$ the collection of interval maps of class~$C^3$ with non-flat critical points, whose Julia set contains at least~$2$ points and is completely invariant.
Note that if~$f : I \to I$ is a non-degenerate smooth map that is topologically exact, then~$J(f) = I$ and~$f$ is in~$\sA_{\R}$.
On the other hand, if~$f$ is an interval map in~$\sA_{\R}$ that is topologically exact on~$J(f)$, then~$J(f)$ has no isolated points.
For more background on the theory of Julia sets, see for example~\cite{dMevSt93} for the real setting, and~\cite{CarGam93,Mil06} for the complex setting.

Throughout the rest of this article we put~$\sA \= \sA_{\R}\cup \sA_{\C}$ and for each~$f$ in~$\sA$ we restrict the action of~$f$ to its Julia set.
In particular, the topological pressure of~$f$ is defined through measures supported on~$J(f)$ and equilibrium states are supported on~$J(f)$.

\begin{defi}
For~$\beta> 0$, a map~$f$ in~$\sA$ satisfies the \emph{Polynomial Shrinking Condition with exponent $\beta$}, if there exist constants~$\rho_0 > 0$ and~$C_0 > 0$ such that for every~$x$ in~$J(f)$, every integer $n \ge 1$, and every connected component~$W$ of $f^{-n}(B(x, \rho))$, we have
$$ \diam(W) \le C n^{-\beta}. $$
\end{defi}

\begin{generic}[Main Theorem]\label{mtm:polyshrink}
Let~$\beta > 1$, and let~$f$ be a map in~$\sA$ satisfying the Polynomial Shrinking Condition with exponent~$\beta$.
Suppose furthermore in the real case that~$f$ is topologically exact on its Julia set.
Then for each~$\alpha$ in~$(\beta^{-1}, 1]$, every H\"{o}lder continuous potential $\varphi: J(f)\to \R$ of exponent~$\alpha$ is hyperbolic for~$f$.
\end{generic}

\begin{rema}
\label{r:TCE}
Recall that a map~$f$ in~$\sA$ satisfies the \emph{Topological Collet-Eckmann Condition}, if there is a constant~$\chi > 0$ such that for every~$\nu$ in~$\cM(J(f), f)$ we have~$\int_{J(f)} \log |Df| \ d\nu \ge \chi$, see~\cite{PrzRivSmi03,Riv1204} for other equivalent formulations.
Each map in~$\sA$ that satisfies the Topological Collet-Eckmann condition satisfies the hypothesis of the Main Theorem for each~$\beta > 1$, see~\cite[Main Theorem']{Riv1204} and~\cite[Proposition~$7.1$]{Riv1206} for the real case, and~\cite[Main Theorem]{PrzRivSmi03} for the complex case.
\end{rema}

\begin{rema}
\label{r:intermittent}
Fix~$\alpha$ in~$(0, 1)$, and consider the map~$f_\alpha$ defined in~\S\ref{ss:statements}.
Well-known arguments show that~$f_{\alpha}$ is topologically exact on~$[0, 1]$, and that~$f_{\alpha}$ satisfies the Polynomial Shrinking Condition with exponent~$\beta = 1 / \alpha$.
Being discontinuous, the map~$f_{\alpha}$ is not in~$\sA$, so we cannot apply the Main Theorem directly to~$f_{\alpha}$.
However, using the Markov structure of~$f_\alpha$ the arguments can be easily adapted to obtain that for each~$\alpha'$ in~$(\alpha, 1]$, every H{\"o}lder continuous potential of exponent~$\alpha'$ is hyperbolic for~$f_{\alpha}$.
The hypothesis that~$\alpha'$ is in~$(\alpha, 1]$ is necessary: The potential~$\varphi_{\alpha} \= - \log |Df_{\alpha}|$ is H{\"o}lder continuous of exponent~$\alpha$, and it is not hyperbolic for~$f_{\alpha}$.
In fact, noting that~$\varphi_{\alpha}(0) = 0$ and~$f_{\alpha}(0) = 0$, for every integer~$n \ge 1$ we have~$\frac{1}{n} S_n(\varphi_{\alpha})(0) = 0$.
Together with the equality~$P(f_{\alpha}, \varphi_{\alpha}) = 0$, shown for example in~\cite[Theorem~$3$]{Lop93}, this implies that~$\varphi_{\alpha}$ is not hyperbolic for~$f_{\alpha}$.
\end{rema}

Given~$f$ in~$\sA$, for each function $\varphi : J(f)\to \R$ and each integer~$n \ge 1$, put
$$ S_n(\varphi)
\=
\varphi + \varphi \circ f + \cdots + \varphi \circ f^{n-1}. $$

\begin{coro}
\label{c:main consequences}
Let~$\beta > 1$, and let~$f$ be a map in~$\sA$ satisfying the Polynomial Shrinking Condition with exponent~$\beta$.
Suppose furthermore in the real case that~$f$ is topologically exact on its Julia set.
Then for every~$\alpha$ in~$(\beta^{-1}, 1]$ and every H{\"o}lder continuous potential~$\varphi$ of exponent~$\alpha$, there is a unique equilibrium state~$\nu$ of~$f$ for the potential~$\varphi$.
Moreover, the measure-theoretic entropy of~$\nu$ is positive and~$\nu$ is exponentially mixing for~$f$.
Finally, for every H{\"o}lder function~$\psi : J(f) \to \R$ of exponent~$\alpha$, the function~$t \mapsto P(f, \varphi + t \psi)$ is real analytic.  
\end{coro}
\begin{proof} 
In the real case, the assertions are a direct consequence of the Main Theorem and of~\cite[Theorem~B]{LiRiv1210a}; for the real analyticity of~$t \mapsto P(f, \varphi + t \psi)$, for each~$t_0$ in~$\R$ apply this result with~$\varphi$ replaced by~$\varphi + t_0 \psi$.

To prove assertions in the complex setting, let~$n \ge 1$ be an integer such that the function~$\tvarphi \= \frac{1}{n} S_n(\varphi)$ satisfies~$\sup_{J(f)} \tvarphi < P(f, \varphi)$.
Since~$f$ is Lipschitz as a self-map of~$\CC$, the function~$\tvarphi$ is H{\"o}lder continuous of exponent~$\alpha$.
Moreover, $\tvarphi$ is co-homologous to~$\varphi$: If we put
$$ h
\=
- \frac{1}{n} \sum_{j = 0}^{n - 1} (n - 1 - j) \varphi \circ f^j, $$
then $\tvarphi = \varphi + h - h \circ f$.
This implies that for every~$\nu$ in~$\cM(J(f), f)$ we have~$\int_{J(f)} \tvarphi \ d\nu = \int_{J(f)} \varphi \ d\nu$, and therefore that~$P(f, \tvarphi) = P(f, \varphi)$ and that~$\tvarphi$ and~$\varphi$ share the same equilibrium states.
An analogous argument shows that if we put~$\tpsi \= \frac{1}{n} S_n(\psi)$, then for every~$t$ in~$\R$ we have~$P(\tvarphi + t \tpsi) = P(\varphi + t \psi)$.
Thus
$$ \sup_{J(f)} \tvarphi < P(f, \varphi) = P(f, \tvarphi), $$
and for each equilibrium state~$\nu$ of~$f$ for the potential~$\varphi$, we have
$$ h_\nu(f)
=
P(f, \varphi) - \int_{J(f)} \varphi \ d\nu
=
P(f, \tvarphi) - \int_{J(f)} \tvarphi \ d\nu
\ge
P(f, \tvarphi) - \sup_{J(f)} \tvarphi
>
0. $$
On the other hand, the fact that~$\nu$ is the unique equilibrium state of~$f$ for the potential~$\varphi$, and that~$\nu$ is exponentially mixing, is obtained by applying~\cite{Hay99} or~\cite{SzoUrbZdu11one} to the potential~$\tvarphi$.
To prove that the function~$t \mapsto P(\varphi + t \psi)$ is real analytic, observe first that, when~$\psi = \varphi$, \cite[Theorem~$47$ in~\S$6$]{SzoUrbZdu11one} asserts the function~$t \mapsto P(\varphi + t \psi)$ is real analytic on a neighborhood of~$t = 0$.
The proof of this result extends easily to the case~$\psi$ is different from~$\varphi$.
To prove that~$t \mapsto P(\varphi + t \psi)$ is real analytic on all of~$\R$, for each~$t_0$ in~$\R$ apply this result with~$\varphi$ replaced by~$\varphi + t_0 \psi$.
This completes the proof that the proof of the corollary.
\end{proof}
\subsection{Proofs of Theorems~\ref{t:real} and~\ref{t:complex} assuming the Main Theorem}
\label{ss:proof of real and complex}
We recall the ``backward contracting property''.

For each map~$f$ in~$\sA$, put
$$ \CV(f) \= f(\Crit(f))
\text{ and }
\Crit'(f) \= \Crit(f)\cap J(f). $$
For a subset~$V$ of~$\dom(f)$, and an integer~$m \ge 1$, each connected component of~$f^{-m}(V)$ is a \emph{pull-back of by~$f^m$}.
For every~$c$ in~$\Crit(f)$ and $\delta >0$, denote by~$\tB(c,\delta)$ the pull-back of~$B(f(c),\delta)$ by~$f$ that contains~$c$.

\begin{defi}
  Given a constant $r>1$, a map~$f$ in~$\sA$ is \emph{backward contracting with constant $r$}, if there exists $\delta_{0}>0$ such that for every~$c$ in~$\Crit'(f)$, every~$\delta$ in~$(0, \delta_0)$, every integer $n\ge 0$, and every component~$W$ of $f^{-n}(\tB(c,r\delta)),$  we have that
$$ \dist (W, \CV(f)) \le \delta
\text{ implies }
\diam(W) < \delta. $$
If for each~$r > 1$ the map~$f$ is backward contracting with constant~$r$, then~$f$ is \emph{backward contracting}.
\end{defi}

A map~$f$ in~$\sA$ is \emph{expanding away from critical points}, if for every neighborhood~$V$ of~$\Crit'(f)$ the map~$f$ is uniformly expanding on the set
$$ K(V)
\=
\{z\in J(f): f^i(z)\not\in V \text{ for all } i\ge 0\}. $$
In other words, there exist $C>0$ and  $\lambda>1$ such that for every~$z$ in~$K(V)$ and~$n\ge 0$, we have $|Df^n(z)|\ge C\lambda^n$.

\begin{fact}[\cite{RivShe1004}, Theorem~A]
\label{f:Polyshrink}
For each map~$f$ in~$\sA$ and each $\beta>0,$ there exists $r>1$ such that the following property holds.
If~$f$ is backward contracting with constant~$r$ and is expanding away from critical points, then~$f$ satisfies the Polynomial Shrinking
Condition with exponent~$\beta$.
\end{fact}

\begin{proof}[Proof of Theorem~\ref{t:real}]
By~\cite[Theorem~$1$]{BruRivShevSt08}, the map~$f$ is backward contracting and by Ma{\~n}{\'e}'s theorem~$f$ is expanding away from
critical points, see for example~\cite{dMevSt93}.
Then Fact~\ref{f:Polyshrink} implies that for each~$\beta > 1$ the map~$f$ satisfies the Polynomial Shrinking Condition with exponent~$\beta$.
So the desired assertions are a direct consequence of the Main Theorem and Corollary~\ref{c:main consequences}.
\end{proof}

\begin{proof}[Proof of Theorem~\ref{t:complex}]
By either~\cite[Theorem~A]{LiShe10b} or~\cite[Theorem~A]{Riv07}, the map~$f$ is backward contracting, and by either~\cite{KozvSt09} or~\cite[Corollary~$8.3$]{Riv07}, it is expanding away from critical points.
Then Fact~\ref{f:Polyshrink} implies that for each~$\beta > 1$ the map~$f$ satisfies the Polynomial Shrinking Condition with exponent~$\beta$.
So the desired assertions are a direct consequence of the Main Theorem and Corollary~\ref{c:main consequences}.
\end{proof}

\subsection{A reduction}
\label{ss:a reduction}
In this section we prove the Main Theorem assuming the following lemma, whose proof occupies \S\S\ref{s:constructing a free IMFS}, \ref{s:proof of Key Lemma}.

\begin{generic}[Key Lemma]
Let~$\beta > 1$, and let~$f$ be a map in~$\sA$ satisfying the Polynomial Shrinking of Components condition with exponent~$\beta$.
In the case~$f$ is an interval map, assume furthermore that~$f$ is topologically exact on its Julia set.
Then for every~$\alpha$ in~$(\beta^{-1}, 1]$, every H\"{o}lder continuous function $\varphi:J(f)\to \R$ of exponent~$\alpha$, and every invariant ergodic probability measure~$\nu$ supported on~$J(f)$, there is a set of full measure of points~$x_0$ such that
$$ \limsup_{n\rightarrow + \infty}\frac{1}{n}\log \sum_{y\in f^{-n}(x_0)}\exp \left( S_n(\varphi)(y) \right)
>
\int_{J(f)} \varphi \ d\nu. $$
\end{generic}

The proof of the Main Theorem is given after the following lemma.
In the case~$f$ is a complex rational map, the following lemma is~\cite[Lemma~$4$]{Prz90}.
The proof applies without change to the case where~$f$ is an interval map.
\begin{lemm}
\label{l:tree pressure}
For each map~$f$ in~$\sA$, and every continuous function~$\varphi : J(f) \to \R$, we have
$$ P(f, \varphi)
\ge
\limsup_{n\rightarrow + \infty}\frac{1}{n}\log \left( \sup_{x_0 \in J(f)} \sum_{y\in f^{-n}(x_0)} \exp (S_n\varphi(y)) \right). $$
\end{lemm}

\begin{proof}[Proof of the Main Theorem assuming the Key Lemma]
We show first
\begin{equation}
\label{e:mean average versus invariant average}
\limsup_{n \rightarrow + \infty} \left( \sup_{J(f)}\frac{1}{n}S_n(\varphi) \right)
\le
\sup_{\nu\in \cM(J(f), f)} \int_{J(f)} \varphi \ d\nu.
\end{equation}
Since~$J(f)$ is compact, for each integer $n\ge 1$ there is a point~$x_n$ of~$J(f)$ such that
$$ S_n(\varphi)(x_n)
=
\sup_{J(f)} S_n(\varphi). $$
Put $\nu_n \= \frac{1}{n}\sum_{i=0}^{n-1}\delta_{f^i(x_n)}$, so that
$$ \int_{J(f)} \varphi \ d\nu_n
=
\frac{1}{n}S_n(\varphi)(x_n)
=
\sup_{J(f)}\frac{1}{n}S_n(\varphi). $$
It follows that there is a sequence positive integers $(n_k)_{k = 1}^{+ \infty}$ such that
$$ \lim_{k\rightarrow + \infty} \int_{J(f)} \varphi \ d\nu_{n_k}
=
\limsup_{n\rightarrow + \infty} \left( \sup_{J(f)}\frac{1}{n}S_n(\varphi) \right). $$
Let~$\nu$ be an accumulation point of~$\nu_{n_k}$ in the weak* topology.
Then~$\nu$ is in~$\cM(J(f), f)$, and
$$ \int_{J(f)} \varphi \ d\nu
=
\limsup_{n\rightarrow + \infty} \left( \sup_{J(f)}\frac{1}{n}S_n(\varphi) \right). $$
This proves~\eqref{e:mean average versus invariant average}.

To prove that~$\varphi$ is hyperbolic for~$f$, let~$\nu_0$ be an invariant probability measure maximizing the function~$\nu \mapsto \int_{J(f)} \varphi \ d \nu$.
Then for almost every ergodic component~$\nu_0'$ of~$\nu_0$, we have~$\int_{J(f)} \varphi \ d \nu_0' = \int_{J(f)} \varphi \ d \nu_0$.
Thus, the Key Lemma applied to such a~$\nu_0'$, together with Lemma~\ref{l:tree pressure} imply
$$ P(f, \varphi)
>
\int_{J(f)} \varphi \ d\nu_0'
=
\int_{J(f)} \varphi \ d\nu_0
=
\sup_{\nu \in \cM(J(f), f)} \int_{J(f)} \varphi \ d\nu. $$ 
Together with~\eqref{e:mean average versus invariant average}, this implies that~$\varphi$ is hyperbolic for~$f$ and completes the proof of the Main Theorem.
\end{proof}

\section{Iterated Multivalued Function Systems}
\label{s:constructing a free IMFS}
This section is devoted to the construction of an ``Iterated Multivalued Function System", which is the main ingredient in the proof of the Key Lemma, compare with~\cite[\S\S$5$, $6$]{InoRiv12}.
It is stated as Proposition~\ref{p:constructing a free IMFS}, below.

Given intervals~$B$ and~$W$ of~$\R$, a \emph{multivalued function $\phi : B \to W$} is a function that to each point~$x$ of~$B$ associates a non-empty subset~$\phi(x)$ of~$W$.
Note that each surjective function~$f : W \to B$ defines a multivalued function~$f^{-1} : B \to W$.
If~$\phi : B \to W$ and~$\tphi : \tB \to \tW$ are multivalued functions such that~$\tW$ is contained in~$B$, then the \emph{composition~$\phi \circ \tphi$ of~$\phi$ and~$\tphi$} is the multivalued function~$\phi \circ \tphi : \tB \to W$ defined for~$x$ in~$\tB$, by~$(\phi \circ \tphi)(x) \= \bigcup_{y \in \tphi(x)} \phi(y)$.
The composition of multivalued functions is easily seen to be associative.

Let~$f$ be a map in~$\sA$.
Given a compact and connected subset~$B_0$ of~$\dom(f)$ intersecting~$J(f)$, a sequence multivalued functions~$(\phi_l)_{l = 1}^{+ \infty}$ is an \emph{Iterated Multivalued Function System (IMFS) generated by~$f$}, if for every~$l$ there is an integer $m_l \ge 1$, and a pull-back~$W_l$ of~$B_0$ by~$f^{m_l}$ contained in~$B_0$, such that
$$ f^{m_l}(W_l) = B_0
\text{ and }
\phi_l = (f^{m_l}|_{W_l})^{-1}. $$
In this case, $(m_l)_{l = 1}^{+ \infty}$ is the \emph{time sequence of $(\phi_l)_{l = 1}^{+ \infty}$}, and \emph{$(\phi_l)_{l = 1}^{+ \infty}$ is defined on~$B_0$}.
Note that for each subset~$A$ of~$B_0$ and each~$l$, the set~$\phi_l(A) \= f^{-m_l}(A)\cap W_l$ is non-empty.

Let~$(\phi_l)_{l = 1}^{+ \infty}$ be an IMFS generated by~$f$ with time sequence~$(m_l)_{l = 1}^{+ \infty}$, defined on a set~$B_0$.
For each integer~$n \ge 1$ put $\Sigma_n \= \{1, 2, \ldots \}^n $ and denote the space of all finite words in the alphabet~$\{1, 2, \ldots, \}$ by $\Sigma^* \= \bigcup_{n\ge 1}\Sigma_n$.
For every integer~$k \ge 1$ and~$\ul = l_1 \cdots l_k$ in~$\Sigma^*$, put
$$ | \ul |=k,
m_{\ul} = m_{l_1}+m_{l_2}+\cdots +m_{l_k},
\text{ and }
\phi_{\ul} = \phi_{l_1}\circ\cdots\circ\phi_{l_k}. $$
Note that for every~$x_0$ in~$B_0$, and every pair of distinct words~$\ul$ and~$\ul'$ in~$\Sigma^*$ satisfying~$m_{\ul} = m_{\ul'}$, we have the following property:
\begin{center}
(*) \qquad
If the sets~$\phi_{\ul}(x_0)$ and~$\phi_{\ul'}(x_0)$ intersect, then they coincide.
\end{center}
The IMFS $(\phi_l)_{l = 1}^{+ \infty}$ is \emph{free}, if there is~$x_0$ in~$B_0$ such that for every pair of distinct words~$\ul$ and~$\ul'$ in~$\Sigma^*$ such that~$m_{\ul}=m_{\ul'}$, the sets~$\phi_{\ul}(x_0)$ and~$\phi_{\ul'}(x_0)$ are disjoint.

\begin{prop}
\label{p:constructing a free IMFS}
Let~$\beta > 1$, and let~$f$ be a map in~$\sA$ satisfying the Polynomial Shrinking of Components condition with exponent~$\beta$.
In the case~$f$ is an interval map, suppose furthermore that~$f$ is topologically exact on its Julia set.
Let~$\alpha$ in~$(\beta^{-1}, 1]$, let~$\varphi: J(f)\to \R$ be H{\"o}lder continuous with exponent~$\alpha$, and let~$\nu$ be an ergodic invariant probability measure on~$J(f)$ that is not supported on a periodic orbit.
Then there exists a subset~$X$ of~$J(f)$ of full measure with respect to~$\nu$, such that for every point~$x_0$ in~$X$ the following property holds: There exist~$D > 0$, a compact and connected subset~$B_0$ of~$\dom(f)$ containing~$x_0$, and a free IMFS~$(\phi_l)_{l=1}^{+ \infty}$ generated by~$f$ with time sequence~$(m_l)_{l=1}^{+ \infty}$, such that~$(\phi_l)_{l=1}^{+ \infty}$ is defined on~$B_0$, and such that for every~$l$ and every~$y$ in~$\phi_l(B_0)$ we have
\begin{equation}\label{e:fre}
S_{m_l}(\varphi)(y)
\ge
m_{l} \int_{J(f)} \varphi \ d \nu - D.
\end{equation}
\end{prop}

The proof of this proposition is at the end of this section.

\begin{lemm}
\label{l:almost properness}
For each interval map~$f : I \to I$ in~$\sA$ there is~$\varepsilon > 0$ such that the following property holds.
Let~$J_0$ be an interval contained in~$I$ satisfying~$|J_0| \le \varepsilon$, let~$n \ge 1$ be an integer, and let~$J$ be a pull-back of~$J_0$ by~$f^n$ whose closure is contained in the interior of~$I$.
Suppose in addition that for each~$j$ in~$\{1, \ldots, n \}$ the pull-back of~$J_0$ by~$f^j$ containing~$f^{n - j}(J)$ has length bounded from above by~$\varepsilon$.
Then~$f^n(\partial J) \subset \partial J_0$.
\end{lemm}
\begin{proof}
  Note that each preimage of a point of~$\partial I$ is either a point of~$\partial I$, or a turning point in the interior of~$I$.
It follows that there is~$\varepsilon$ in~$(0, |I|)$ such that for every interval~$J'$ contained in~$I$ that shares an endpoint with~$I$ and satisfies~$|J'| \le \varepsilon$, the following property holds: For each pull-back~$J$ of~$J'$ by~$f$, either~$J$ shares an endpoint with~$I$ and~$f : J \to J'$ is a bijection, or both endpoints of~$J$ are mapped to the endpoint of~$J'$ that is not an endpoint of~$I$.
Reducing~$\varepsilon$ if necessary, assume that for every pair of distinct elements~$c$ and~$c'$ of~$\Crit(f) \cup \partial I$, we have~$\varepsilon < \dist(c, c')$.
We prove the lemma for this choice of~$\varepsilon$.

Let~$J_0$, $n$, and~$J$ be as in the statement of the lemma, and for each~$j$ in~$\{1, \ldots, n - 1 \}$ let~$J_j$ be the pull-back of~$J_0$ by~$f^j$ containing~$f^{n - j}(J)$.
Note that~$J_n = J$, and that our hypotheses imply that for each~$j$ in~$\{0, \ldots, n \}$ we have~$|J_j| \le \varepsilon$.
We prove by induction that for each~$j$ in~$\{0, \ldots, n \}$, either~$f^j(\partial J_j) \subset \partial J_0$, or~$J_j$ shares an endpoint with~$I$ and~$f^j$ maps the endpoint of~$J_j$ that is not an endpoint of~$I$ to an endpoint of~$J_0$; the lemma follows by taking~$j = n$.
The case~$j = 0$ being trivial, let~$j$ in~$\{1, \ldots, n \}$ be such that this property holds with~$j$ replaced by~$j - 1$.
If~$J_{j - 1}$ shares an endpoint with~$I$, then the desired assertion follows from the induction hypothesis and our choice of~$\varepsilon$.
Suppose~$J_{j - 1}$ does not share an endpoint with~$I$.
If~$J_j$ contains a turning point in its interior, then by our choice of~$\varepsilon$ the interval~$J_j$ does not contain any other turning point of~$f$.
It follows that both endpoints of~$J_j$ are mapped to the same end point of~$J_{j - 1}$ by~$f$.
So by the induction hypothesis we have~$f^j(\partial J_j) \subset \partial J_0$.
It remains to consider the case where~$J_j$ does not contain a turning point of~$f$ in its interior.
Then~$f$ is injective on~$J_j$.
Thus, either~$f : J_j \to J_{j - 1}$ is a bijection, or~$J_j$ shares an endpoint with~$I$.
In the former case we have~$f^j(\partial J_j) = f^{j - 1}(\partial J_{j - 1}) \subset \partial J_0$.
In the latter case, $f$ maps the endpoint~$x_j$ of~$J_j$ that is not an endpoint of~$I$ to an endpoint of~$J_{j - 1}$, so by the induction hypothesis~$f^j(x_j)$ is in~$\partial J_0$.
This completes the proof of the induction step and of the lemma.
\end{proof}

\begin{lemm}
\label{l:one-sided covering}
Let~$f$ be an interval map in~$\sA$ that is topological exact on its Julia set, and let~$x_0$ be a point of~$J(f)$ such that~$(x_0, +\infty)$ (resp. $(-\infty, x_0)$) intersects~$J(f)$. 
Then for every open interval~$U$ intersecting~$J(f)$, and every sufficiently large integer~$n \ge 1$, there is a point~$y$ of~$U$ in~$f^{-n}(x_0)$ such that for every $\eps>0$ the set $f^n(B(y,\varepsilon))$ intersects~$(x_0, +\infty)$  (\text{resp.} $(-\infty, x_0)$).
\end{lemm}
\begin{proof}
Using that~$f$ is topological exact on its Julia set $J(f),$ we know that there is an integer $N\geq 1$ such that for every $n\geq N$ we have $f^n(U)\supset J(f)$.
Fix $n\geq N$.
Note that the set~$f^{-n}(x_0)$ is finite, and there is a point~$z$ of~$U$ such that~$f^n(z)$ is in $(x_0, +\infty)$ (resp. $(-\infty, x_0)$).
Let~$y$ be a point of $f^{-n}(x_0)$ in~$U$ such that for every $y'$ of $f^{-n}(x_0)$ in~$U$ we have $|y-z|\leq |y'-z|.$ Now let us prove the lemma holds for such $y.$ In fact, otherwise, there is $\eps_0\in (0, |y-z|)$ such that~$f^n(B(y,\eps_0))$ is contained in~$(-\infty, x_0]$ (resp. $[x_0, +\infty)$).
It follows that there is a point~$z'$ of $B(y, \eps_0)\cap U$ such that~$f^n(z')$ is in $(-\infty, x_0)$ (resp. $(x_0, +\infty)$) and $|z'-z|<|y-z|$.
Since $f^n$ is continuous on~$U$ it follows that there is $y''$ in $U$ such that $|y''-z|<|y-z|$ and $f^n(y'')=x_0.$ This is a contradiction with our choice of $y.$ The lemma is proved.
\end{proof}

\begin{lemm}
\label{l:bounded distortion}
Let~$\beta > 1$, and let~$f$ be a map in~$\sA$ satisfying the Polynomial Shrinking Condition with exponent~$\beta$.
Then for every~$\alpha$ in~$(\beta^{-1}, 1]$ and every H{\"o}lder continuous function $\varphi: J(f)\rightarrow \R$ of exponent~$\alpha$ there exist constants $\rho_1 > 0$ and~$C_1 > 1$, such that for every point~$z$ of~$J(f)$, every integer~$n\ge 1$, and every pull-back~$W$ of~$B(z,\rho_1)$ by~$f^n$ the following holds: For every~$x$ and~$x'$ in~$W$ we have
$$ |S_{n}(\varphi)(x)-S_{n}(\varphi)(x')|\le C_1. $$
\end{lemm}
\begin{proof}
Let~$C_* \ge 1$ be such that for every~$z$ and~$z'$ in~$J(f)$ we have
$$ |\varphi(z)-\varphi(z')|
\le
C_* \dist (z, z')^\alpha. $$
Noticing that $f$ satisfies the Polynomial Shrinking Condition with
exponent~$\beta$, there exist constants~$\rho_0 > 0$ and~$C_0 > 1$ such that for every point~$z$ in~$J(f)$, every integer $n \ge 1$, and every pull-back~$W'$ of~$B(z,\rho_0)$ by~$f^n$, we have $\diam (W') \le C_0 n^{-\beta}$.
Therefore, for every integer~$n \ge 1$, every point~$z$ in~$J(f)$ and every pair of points~$x$ and~$x'$ in the same pull-back~$W$ of~$B(z, \rho_0)$ by~$f^n$, we have
\begin{align*}
|S_n(\varphi)(x) - S_n(\varphi)(x')|
& =
\left| \sum_{i=0}^{n-1} \varphi(f^i(x))- \sum_{i=0}^{n-1}\varphi(f^i(x')) \right|
\\ &
\le
\sum_{i=0}^{n-1} C_*\diam (f^i(W))^{\alpha}
\\ & \le
\sum_{i=0}^{n-1}C_* C_0^{\alpha} (n-i)^{- \beta \alpha}
\\ & \le
C_*C_0^{\alpha}\sum_{m = 1}^{+ \infty} m^{- \beta\alpha}.
\end{align*}
This proves the lemma with constants~$\rho_1 = \rho_0$ and~$C_1 = C_*C_0^{\alpha}\sum_{m = 1}^{+ \infty} m^{-\beta\alpha}$.
\end{proof}

\begin{lemm}
\label{l:maximizing branches}
Let~$f$ be a map in~$\sA$, let $\varphi: J(f)\rightarrow \R$ be a continuous function, and let~$\nu$ be an invariant and ergodic Borel probability measure supported on~$J(f)$.
Then there exists a Borel subset~$X'$ of~$J(f)$ of full measure with respect to~$\nu$, such that the following holds.
For every~$D' > 0$, and every point~$x$ of~$X'$ there exist a strictly increasing sequence positive integers $(n_l)_{l = 1}^{+ \infty}$ such that for every~$l$ we can choose a point~$x_l$ in~$f^{-n_l}(x)$ so that:
\begin{enumerate}
\item[1.]
$x_{l + 1}$ is in~$f^{-(n_{l + 1} - n_l)}(x_l)$;
\item[2.]
$S_{n_l}(\varphi)(x_l)\ge n_l\int_{J(f)} \varphi \ d\nu -D'$.
\end{enumerate}
\end{lemm}

To prove this lemma we use the natural extension of~$f|_{J(f)}$, that we proceed to recall.
Let $\Z_{-}$ denote the set of all non-positive integers and endow
$$ \cZ
\=
\left\{ (z_m)_{m\in \Z_{-}} \in J(f)^{\Z_{-}} : \text{ for every } m \in \Z_-, f(z_{m - 1})=z_m \right\} $$
with the product topology.
Define $T: \cZ\rightarrow \cZ$ by
$$ T \left( (\cdots,z_{-2},z_{-1},z_0) \right)
=
(\cdots,z_{-2},z_{-1},z_0, f(z_0)) $$
and $\pi: \cZ\rightarrow J(f)$ by $\pi((z_m)_{m\in \Z_{-}}) = z_0$.
Note that~$T$ is a bijection, $T^{-1}$ is measurable, $\pi$ is continuous and onto, and~$\pi\circ T= f\circ \pi$.
We call~$(\cZ,T)$ the \emph{natural extension of $(J(f), f)$}.
If~$\nu$ is a Borel probability measure~on~$J(f)$ that is invariant and ergodic for~$f$, then there exists a unique Borel probability measure~$\tnu$ on~$\cZ$ that is invariant and ergodic for~$T$, and that satisfies~$\pi_* \tnu=\nu$, see for example~\cite[\S$2.7$]{PrzUrb10}.

The following is a well-known consequence of the pointwise ergodic theorem, see for example~\cite[Lemma~$1.3$]{PrzRivSmi04} for a proof.
\begin{lemm}\label{lem:ergodic}
Let $(\cZ, \sB, \tnu)$ be a probability space, and let $T : \cZ \rightarrow \cZ$ be an ergodic measure preserving transformation.
Then for each function $\psi: \cZ \rightarrow \R$ that is integrable with respect to~$\tnu$, there exists a subset~$Z$ of~$\cZ$ such that~$\tnu(Z) = 1$, and such that for every~$\underline{z}$ in~$Z$ we have
$$ \limsup_{n\rightarrow + \infty}\sum_{i=0}^{n-1} \left(\psi(T^i(\underline{z})) - \int_{\cZ} \psi \ d\tnu \right)
\ge
0. $$
\end{lemm}

\begin{proof}[Proof of Lemma~\ref{l:maximizing branches}]
Let~$(\cZ, T)$ and~$\tnu$ be as above, and note that~$\tnu$ is also ergodic with respect to~$T^{-1}$.
Applying Lemma~\ref{lem:ergodic} to the continuous function~$\psi = \varphi \circ \pi$, we obtain that there exists a subset~$Z$ of~$\cZ$ of full measure with respect to~$\tnu$, such that for every point~$(z_m)_{m \in \Z_{-}}$ in~$Z$ we have
$$\limsup_{n\rightarrow + \infty}\sum_{i=0}^{n-1}\left(\varphi\circ\pi \left( T^{-i} \left( (z_m)_{m \in \Z_{-}} \right) \right) - \int_{\cZ} \varphi\circ\pi \ d\tnu\right)
\ge
0. $$
Note that the set $X' \= \pi(Z)$ satisfies
$$ \nu(X') = \tnu \left( \pi^{-1}(\pi (Z)) \right) \ge \tnu(Z) = 1, $$
so~$\nu(X') = 1$.
To verify that~$X'$ satisfies the desired properties, let~$D' > 0$, let~$x$ be a point of~$X'$, and choose a point~$(z_m)_{m \in \Z_{-}}$ of~$Z$ such that~$\pi \left( (z_m)_{m\in \Z_{-}} \right) = x$.
Then there is a strictly increasing sequence positive integers $(n_l)_{l = 1}^{+ \infty}$ such that for every~$l$ we have
\begin{equation}
\label{e:erg}
\sum_{i=0}^{n_l - 1} \varphi \circ \pi \left( T^{-i} \left( (z_m)_{m \in \Z_{-}} \right) \right)
\ge
n_l\int_{J(f)} \varphi \ d\nu - D'.
\end{equation}
For each integer~$l \ge 1$ put
$$ x_l
\=
\pi \left( T^{-n_l} \left( (z_m)_{m\in \Z_{-}} \right) \right)
=
z_{n_l}
\in
f^{-n_l}(z). $$
The first part is then a direct consequence of the definitions, and the second follows from~\eqref{e:erg}.
\end{proof}

\begin{proof}[Proof of Proposition~\ref{p:constructing a free IMFS}]
In the real case, denote by~$I$ the domain of~$f$.

Let~$X'$ be the subset of~$J(f)$ given by Lemma~\ref{l:maximizing branches}, and let~$X$ be the complement in~$X'$ of the set of preperiodic points of~$f$.
Since~$\nu$ is ergodic and it is not supported on a periodic orbit, the set~$X$ has full measure for~$\nu$.
Fix a point~$x_0$ of~$X$.
In the real case assume~$x_0$ is not an endpoint of~$I$.

Let~$\varepsilon > 0$ be the constant given by Lemma~\ref{l:almost properness}, and let~$\rho_1 > 0$ and~$C_1 > 1$ be the constants given by Lemma~\ref{l:bounded distortion}.
Moreover, let~$\rho_0 > 0$ and~$C_0 > 1$ be such that for every~$z$ in~$J(f)$, every integer $n \ge 1$, and every pull-back~$W$ of~$B(z,\rho_0)$ by~$f^n$, we have
$$ \diam (W)
\le
\min \{ C_0 n^{-\beta}, \varepsilon \}. $$
Fix~$\rho$ in~$\left(0, \min \{\rho_0, \rho_1 \} \right)$.
In the real case, assume in addition that~$\rho < \dist(x_0, \partial I)$.

In part~$1$ below we define the IMFS, and in part~$2$ we show it is free and that it satisfies~\eqref{e:fre}.

\partn{1}
Let~$D'$, $(n_l)_{l = 1}^{+ \infty}$, and~$(x_l)_{l = 1}^{+ \infty}$ be given by Lemma~\ref{l:maximizing branches} with~$x = x_0$.
Taking a subsequence if necessary, assume~$(x_l)_{l = 1}^{+ \infty}$ converges to a point~$w_0$.
In the complex case, using that~$f$ is topologically exact on~$J(f)$, there exist an integer~$M \ge 1$ and distinct points~$y_0$ and~$y_1$ of~$f^{-M}(w_0)$ in~$B(x_0, \rho)$.
Let~$\rho'$ in $(0, +\infty)$ be such that the pull-backs~$U_0$ and~$U_1$ of~$B(w_0,\rho')$ by~$f^M$ containing~$y_0$ and~$y_1$, respectively, are disjoint and contained in~$B(x_0, \rho)$.
Note that, since in the complex case the map~$f$ is open, for every sufficiently large~$l$ the point~$x_l$ is contained in~$f^M(U_0)$ and in~$f^M(U_1)$.
In the real case, assume without loss of generality that for infinitely many integers~$l$ the point~$x_l$ satisfies~$x_l > w_0$.
Note that~$x_0$ is not in the boundary of a periodic Fatou component, since by hypothesis~$x_0$ is not preperiodic.
Therefore we can assume that there are~$2$ disjoint open intervals~$\tU_0$ and~$\tU_1$ in~$(x_0 - \rho, x_0)$, each of them intersecting~$J(f)$.
Since~$f$ is topologically exact on~$J(f)$ and~$w_0$ is in~$J(f)$, by Lemma~\ref{l:one-sided covering} there is an integer~$M \ge 1$, and distinct points~$y_0$ and~$y_1$ of~$f^{-M}(w_0)$ in~$\tU_0$ and~$\tU_1$, respectively, such that for every~$\eps>0$ each of the sets~$f^M(B(y_0, \eps))$ and~$f^M(B(y_1, \eps))$ intersects~$(w_0, +\infty)$.
Let~$\rho'>0$ be such that the pull-backs~$U_0$ and~$U_1$ of~$B(w_0,\rho')$ by~$f^M$ containing~$y_0$ and~$y_1$, respectively, are contained in~$\tU_0$ and~$\tU_1$, respectively.
It follows that~$U_0$ and~$U_1$ are disjoint and contained in~$(x_0 - \rho, x_0)$ and that for infinitely many~$l$ the point~$x_l$ is contained in~$f^M(U_0)$ and in~$f^M(U_1)$.
Using the Polynomial Shrinking Condition and taking a subsequence if necessary, assume in both, the real and complex cases, that for every integer~$l \ge 1$ we have~$n_{l + 1} - n_l \ge M$, that the point~$x_l$ is contained in~$f^M(U_0)$ and in~$f^M(U_1)$, and that the pull-back~$W_l$ of~$\overline{B(x_0,\rho)}$ by~$f^{n_l}$ containing~$x_l$ is contained in~$B(w_0,\rho')$.
Interchanging~$y_0$ and~$y_1$, and taking a subsequence if necessary, assume that for every~$l$ the point~$f^{n_{l + 1} - n_l - M}(x_{l + 1})$ is not in~$U_0$.
For each~$l$ choose a pull-back~$W_l'$ of~$W_l$ by~$f^M$ that contains a point~$x_l'$ of~$f^{-M}(x_l)$ and that is contained in~$U_0$.

In the complex case put
$$ B_0 \= \overline{B(x_0, \rho)},
M' \= M,
U_0' \= U_0, $$
and for each~$l$ put~$W_l'' \= W_l'$.
Note that for each~$l$ we have
$$ W''_l \subset U_0' \subset B_0
\text{ and }
f^{n_l + M'}(W_l'') = B_0. $$

In the real case, note that~$W_l'$ is contained in~$U_0 \subset B(x_0, \rho)$, so the closure of~$W_l'$ is contained in the interior of~$I$.
So by Lemma~\ref{l:almost properness} the set~$f^{n_l + M}(\partial W_l')$ is contained in~$\partial B(x_0, \rho)$.
Thus, for each~$l$ the set~$f^{n_l + M}(W_l')$ contains either~$[x_0 - \rho, x_0]$ or~$[x_0, x_0 + \rho]$.
Suppose there are infinitely many~$l$ such that~$f^{n_l + M}(W_l')$ contains~$[x_0 - \rho, x_0]$.
Taking a subsequence if necessary, assume this holds for every~$l$.
Then for every~$l$ there is a pull-back~$W_l''$ of~$[x_0 - \rho, x_0]$ by~$f^{n_l + M}$ that is contained in~$W_l'$ and such that~$f^{n_l + M}(W_l'') = [x_0 - \rho, x_0]$.
In this case we put
$$ B_0 \= [x_0 - \rho, x_0],
M' \= M,
\text{ and }
U_0' \= U_0, $$
and note that~$W_l'' \subset W_l' \subset U_0' \subset [x_0 - \rho, x_0] = B_0$.
It remains to consider the case where for each~$l$, outside finitely many exceptions, the set~$f^{n_l + M}(W_l')$ contains~$[x_0, x_0 + \rho]$, but it does not contain~$[x_0 - \rho, x_0]$.
Taking a subsequence if necessary, assume this holds for every~$l$.
Since~$x_0$ is not in the boundary of a Fatou component, by Lemma~\ref{l:one-sided covering} there is an integer~$\tM \ge 1$ and a pull-back~$U_0'$ of~$U_0$ by~$f^{\tM}$ that is contained in~$(x_0, x_0 + \rho)$, and such that for infinitely many~$l$ the point~$x_l'$ is contained in~$f^{\tM}(U_0')$.
Taking a subsequence if necessary, assume that for every~$l$ we have~$n_{l + 1} - n_l \ge M + \tM$, and that the point~$x_l'$ is contained in~$f^{\tM}(U_0')$.
Since for each~$l$ the point~$f^{n_{l + 1} - n_l - M}(x_{l + 1})$ is not in~$U_0$, it follows that the point~$f^{n_{l + 1} - n_l - M - \tM}(x_{l + 1})$ is not in~$U_0'$.
For each~$l$ choose a pull-back~$\tW_l'$ of~$W_l'$ by~$f^{\tM}$ contained in~$U_0'$ and that contains a point of~$f^{- \tM}(x_l')$.
By Lemma~\ref{l:almost properness}, the set~$f^{n_l + M + \tM}(\partial \tW_l')$ is contained in~$\partial B(x_0, \rho)$.
Since the set~$f^{n_l + M + \tM}(\tW_l')$ is contained in~$f^{n_l + M}(W_l')$ and this last set does not contain~$[x_0 - \rho, x_0]$, we conclude that~$f^{n_l + M + \tM}$ maps both endpoints of~$\tW_l'$ to~$x = x_0 + \rho$.
Since by construction~$f^{n_l + M + \tM}(\tW_l')$ contains~$x = x_0$, we conclude that~$f^{n_l + M + \tM}(\tW_l')$ contains~$[x_0, x_0 + \rho]$.
So there is a pull-back~$W_l''$ of~$[x_0, x_0 + \rho]$ by~$f^{n_l + M + \tM}$ that is contained in~$\tW_l'$, and such that~$f^{n_l + M + \tM}(W_l'') = [x_0, x_0 + \rho]$. 
Note that~$W_l'' \subset \tW_l' \subset U_0' \subset (x_0, x_0 + \rho)$.
In this case we put~$B_0 \= [x_0, x_0 + \rho]$, and~$M' \= M + \tM$.

In both, the real and complex cases, for each~$l$ we put
$$ \phi_l \= \left( f^{n_l + M'}|_{W_l''} \right)^{-1}. $$
Then, $(\phi_l)_{l = 1}^{+ \infty}$ is an IMFS generated by~$f$ with time sequence~$(m_l)_{l=0}^{+ \infty} \= (n_l + M')_{l=0}^{+ \infty}$ that is defined on $B_0$.
Moreover, for each~$l$ we have
$$ n_{l + 1} - n_l \ge M',
W_l'' \subset U_0',
\text{ and }
f^{n_{l + 1} - n_l - M'}(x_{l + 1}) \not \in U_0'. $$

\partn{2}
To prove that the IMFS~$(\phi_l)_{l = 1}^{+ \infty}$ is free, let~$k \ge 1$ and $k'\geq 1$ be integers and let
$$ \ul\=l_1 l_2\cdots l_k
\text{ and }
\ul'\=l'_1 l'_2\cdots l'_{k'} $$
be different words in~$\Sigma^*$ such that~$m_{\ul} = m_{\ul'}$.
Assume without loss of generality that $l_{k'}' \ge l_{k} + 1$.
Note that the set
$$ f^{m_{\ul}-m_{l_k}}(\phi_{\ul}(x_0))
=
\phi_{l_k}(x_0). $$
is contained~$W_{l_k}''$, and therefore in~$U_0'$.
On the other hand, we have
$$ m_{l'_{k'}} - m_{l_k}
=
n_{l'_{k'}} - n_{l_k}
\ge
n_{l_k + 1} - n_{l_k}
\ge
M' $$
and therefore the set
\begin{multline*}
f^{m_{\ul} - m_{l_k}}(\phi_{\ul'}(x_0))
=
f^{m_{\ul'} - m_{l_k}}(\phi_{\ul'}(x_0))
=
f^{m_{l'_{k'}}-m_{l_k}}(\phi_{l'_{k'}}(x_0))
\\ =
f^{m_{l'_{k'}}-m_{l_k} - M'} \left( \left(f^{n_{l_{k'}'}}|_{W_{l_{k'}'}} \right)^{-1}(x_0) \right)
\end{multline*}
contains the point
$$ f^{m_{l_{k'}'} - m_{l_k} - M'} (x_{l_{k'}'})
=
f^{n_{l_{k'}'} - n_{l_k} - M'} (x_{l_{k'}'})
=
f^{n_{l_k + 1} - n_{l_k} - M'} (x_{l_k + 1}). $$
By construction this point is not in~$U_0'$, so we conclude that the sets
$$ f^{m_{\ul} - m_{l_k}}(\phi_{\ul}(x_0))
\text{ and }
f^{m_{\ul} - m_{l_k}}(\phi_{\ul'}(x_0)) $$
are different.
This implies that the sets~$\phi_{\ul}(x_0)$ and~$\phi_{\ul'}(x_0)$ are different, and by property~$(*)$ stated above the statement of the proposition, that they are disjoint.
This completes the proof that the IMFS $(\phi_l)_{l = 1}^{+ \infty}$ is free.

Finally, let us check inequality~\eqref{e:fre} in the statement of the proposition.
Recall that for every~$l$ and~$y$ in~$\phi_l(B_0)$, the point~$f^{M'}(y)$ is in~$W_l$.
Thus, by Lemma~\ref{l:bounded distortion} and by part~$2$ of Lemma~\ref{l:maximizing branches} we have
\begin{align*}
S_{m_l}(\varphi)(y)
& =
S_{n_l}(\varphi)(f^{M'}(y)) + S_{M'}(\varphi)(y)
\\ & \ge
S_{n_l}(\varphi)(x_l) - C_1 + S_{M'}(\varphi)(y)
\\ &
\ge n_l\int_{J(f)} \varphi \ d \nu - D' - C_1 +S_{M'}(\varphi)(y)
\\ & \ge
m_l\int_{J(f)} \varphi \ d \nu -D' - C_1 - 2M' \sup_{J(f)} \varphi.
\end{align*}
This proves~\eqref{e:fre} with $D = D'  + C_1 + 2 M' \sup_{J(f)} \varphi$, and completes the proof of the proposition.
\end{proof}

\section{Proof of the Key Lemma}
\label{s:proof of Key Lemma}
In this section we complete the proof of the Key Lemma.
The case where the measure~$\nu$ is supported on a periodic orbit is different.
For the complex case we refer to~\cite[Proposition~$4.1$]{InoRiv12}.
The proof of~\cite[Proposition~$4.1$]{InoRiv12} does not apply directly to interval maps, as it uses that complex rational maps are open as maps acting on~$\CC$.
The case of interval maps is treated in Lemma~\ref{l:periodic case}, below.
The proof of the Key Lemma is completed after this lemma.

\begin{lemm}
\label{l:periodic case}
Let~$\beta > 1$, and let~$f$ be an interval map in~$\sA$ that is topologically exact on its Julia set, and that satisfies the Polynomial Shrinking of Components condition with exponent~$\beta$.
Then for every $\alpha$ in $(\beta^{-1}, 1]$, every H{\"o}lder continuous function $\varphi: J(f)\to \R$ with exponent $\alpha$, and every periodic point~$x_0$
in $J(f)$  of~$f$ of period~$N$, we have
\begin{equation}
\label{e:periodic gap pressure}
\limsup_{n\rightarrow + \infty}\frac{1}{n}\log \sum_{y\in f^{-n}(x_0)} \exp (S_n(\varphi)(y))
>
\frac{1}{N} S_N(\varphi)(x_0).
\end{equation}
\end{lemm}
\begin{proof}
Let~$\rho_0 > 0$ and~$C_0 > 1$ be such that for every~$z$ in~$J(f)$, every integer $n \ge 1$, and every pull-back~$W$ of~$B(z,\rho_0)$ by~$f^n$, we have~$\diam (W)\le C_0 n^{-\beta}$.
The proof is divided in~$2$ parts.
In part~$1$ we construct an induced map, and in part~$2$ we show an inequality analogous to~\eqref{e:periodic gap pressure} for the induced map, from
which~\eqref{e:periodic gap pressure} follows as a direct consequence.

\partn{1}
Fix a periodic point~$x_0$ in~$J(f)$ of period~$N$.
Since~$|(f^N)'(x_0)|\geq 1$, we know that there is~$\rho$ in~$(0, \rho_0)$ such that there is a local inverse~$\phi$ of~$f^{2N}$ defined on~$B(x_0, \rho)$ and fixing~$x_0$.
Note that~$f^{2N}\circ \phi$ is the identity map on~$B(x_0,\rho)$, hence~$\phi$ is increasing on~$B(x_0,\rho)$, and~$f^{2N}$ is increasing on~$\phi(B(x_0,\rho))$.
Since~$x_0$ is in $J(f)$, changing orientation and reducing~$\rho$ if necessary,
assume that~$(x_0,x_0+\rho/2)$ intersects~$J(f)$, and that for every~$y$ in~$(x_0, x_0 + \rho)$ we have~$\phi(y) < y$.
Since~$f$ is topologically exact on its Julia set, by Lemma~\ref{l:one-sided covering} there exist an integer $k'\geq 1$, and a point~$z'$ in $(x_0, x_0+\rho/2)$, such that~$f^{2Nk'}(z') = x_0$, and such that for every~$\varepsilon > 0$ the set~$f^{2Nk'}(B(z', \varepsilon))$ intersects~$(x_0, x_0 + \rho/2)$.
Fix~$\varepsilon$ in~$(0, |z'-x_0|)$ such that~$f^{2Nk'}(B(z', \varepsilon)) \subset B(x_0,\rho/2)$.
Note that the closure of $B(z',\varepsilon)$ is contained in
$(x_0, x_0+\rho).$

Let~$W$ be the pull-back of~$f^{2Nk'}(B(z', \varepsilon))\cap [x_0, x_0+\rho/2)$ by~$f^{2Nk'}$ containing~$z'$.
Since $f^{2Nk'}$, and hence $\phi^{k'}$, is continuous, reducing~$\varepsilon$ if necessary, assume that
$ U_0' \= \phi^{k'} \left(f^{2Nk'} (W) \right) $
is disjoint from~$\overline{W}$.
By our choice of~$\phi$, and the hypothesis that~$f$ satisfies the Polynomial Shrinking of Components condition, we know that for every $x$ in $(x_0, x_0+\rho)$ we
have $\lim_{k\to +\infty} \diam (\phi^k([x_0, x]))=0$.
This implies that for every interval $U\subset [x_0, x_0+\rho)$ we have 
\begin{equation*}
\lim_{k\to +\infty}\diam(\phi^k(U))=0
\text{ and }
\lim_{k\to +\infty} \dist(\phi^k(U), x_0)=0.
\end{equation*}
Noting that
$$ W\subset B(z',\varepsilon)\subset (x_0, x_0+\rho)
\text{ and }
x_0 \in f^{2Nk'}(W)\subset [x_0, x_0+\rho/2), $$
it follows that there is $k_1\geq 0$ such that
\begin{equation}
\label{e:secondary domain}
U_1 \= \phi^{k_1} \left( W \right) \subset f^{2Nk'}(W),
\end{equation}
and
\begin{equation}
\label{e:secondary domain diameter}
\diam (\phi^{k_1+k'}(f^{2Nk'}(W)))<\diam(f^{2Nk'}(W)).
\end{equation}
Put
$ k_0 \= k_1+k'
\text{ and }
U_0\=\phi^{k_1}(U_0'). $
Then we have
$$ k_0\geq 1,
U_0 \cap U_1 = \emptyset,
\text{ and }
U_1 \subset f^{2Nk'} \left( W \right). $$
By~\eqref{e:secondary domain diameter} and the fact that~$f^{2Nk'} \left(W \right)$ contains~$x_0$, the set
$$ U_0 = \phi^{k_1}(U_0') = \phi^{k_0} \left( f^{2Nk'} \left( W \right) \right) $$
is contained in~$f^{2Nk'} \left( W \right)$.
Finally, note that
$$ f^{2Nk_0}(U_1)
=
f^{2Nk'} \left( W \right)
=
f^{2Nk_0}(U_0). $$
Put
$$ U \=U_0\cup U_1
\text{ and }
\whf \= f^{2Nk_0}|_U. $$

\partn{2}
Put $\hvarphi\= \frac{1}{2Nk_0} S_{2Nk_0}(\varphi)$, for every integer $m\geq 1$ put
$$ \hS_m(\hvarphi)
\=
\hvarphi +\hvarphi\circ \whf+\cdots +\hvarphi\circ \whf^{m-1}, $$
and note that to prove the lemma it suffices to show
\begin{equation}
  \label{e:induced gap pressure}
\limsup_{m\rightarrow + \infty}\frac{1}{m}\log \sum_{y\in \whf^{-m}(x_0)} \exp (\hS_m(\hvarphi)(y))
>
\hvarphi(x_0).
\end{equation}
This is equivalent to show that the radius of convergence of the series
$$ \Xi(s)
\=
\sum_{m = 0}^{+ \infty} \left( \sum_{z \in \whf^{-m}(x_0)} \exp \left( \hS_m(\hvarphi)(z) \right) \right) s^m $$
is strictly less than~$\exp(- \hvarphi(x_0))$.
The proof of this fact is similar to Case~$1$ of the proof of~\cite[Proposition~$4.1$]{InoRiv12}.
We include it here for completeness.

Put $\hK\= \bigcap_{i= 0}^{+ \infty} \whf^{-i}(U)$ and observe that~$x_0$ is contained in this set.
Consider the itinerary map
$$ \iota: \hK \to \{0,1\}^{\{1, 2, \ldots \}} $$
defined so that for every~$i$ in~$\{1, 2, \ldots \}$ the point~$\whf^i(z)$ is in~$U_{\iota(z)_i}$.
Since~$\whf$ maps each of the sets~$U_0$ and~$U_1$ onto $f^{2Nk'} \left( \overline{B(z', \varepsilon)} \right)$, and both of~$U_0$ and~$U_1$ are contained in this
set, for every integer~$k\geq 0$ and every sequence $a_0,a_1, \ldots, a_k$ of elements of~$\{0,1\}$ there is a point of~$\whf^{-(k+1)}(x_0)$ in the set
$$ \hK(a_0 a_1 \cdots a_k)
\=
\left\{z\in \hK: \text{ for every $i$ in $\{0,1,\cdots, k\}$ we have } \iota(z)_i=a_i\right\}. $$
By the Polynomial Shrinking of Components condition and our choice of~$\phi$ and~$U_0$, there is a constant~$\hC>0$ such that for
every integer $k\geq 1$ and every point~$z$ in $\hK(\underbrace{0\cdots0}_k)$, we have
\begin{equation}\label{eq:b1}
\hS_k(\hvarphi)(z) \ge k\hvarphi(x_0) - \hC.
\end{equation}
Taking~$\hC$ larger if necessary, assume that for every point~$z$ in~$U$ we have
 \begin{equation}\label{e:2}
\hvarphi(z)
\ge
\hvarphi(x_0) - \hC.
\end{equation}
It follows that for every~$k\geq 0$ and every sequence $a_0, a_1, \cdots, a_k$ of elements of $\{0,1\}$ with $a_0=1$ and every point $x$ in $\hK(a_0a_1\cdots
a_k),$ we have
\begin{equation}
\label{eq:b2}
\hS_{k+1}(\hvarphi)(x)\geq (k+1)\hvarphi(x_0)-2(a_0+a_1+\cdots+a_k)\hC.
\end{equation}
In fact, put $\ell\= a_0+\cdots +a_k$, $i_{\ell+1}\=k+1$, and let
$$ 0 = i_1 < i_2 < \cdots < i_{\ell} \le k $$
be all integers~$i$ in $\{0,1,\cdots,k\}$ such that $a_i=1$.
Then by~\eqref{eq:b1} and ~\eqref{e:2} for every $j\in \{ 1,\cdots, \ell\}$ we have
$$ \hS_{i_{j+1}-i_j}(\hvarphi)(\whf^{i_j}(x))
\geq
(i_{j+1}-i_j)\hvarphi(x_0)-2\hC. $$
Summing over $j$ in $\{1,2,\cdots, \ell\}$ we obtain~\eqref{eq:b2}.
Thus, if we put
$$ \Phi(s) \= \sum_{k = 1}^{+ \infty} \exp (k \hvarphi(x_0) - 2 \hC) s^k, $$
then each of the coefficients of
$$ \Upsilon(s) \= \Phi(s) + \Phi(s)^2 + \cdots  $$
is less than or equal to the corresponding coefficient of~$\Xi$, and therefore the radius of convergence of~$\Xi$ is less than or equal to that of~$\Upsilon$.
Since clearly~$\Phi(s) \to + \infty$ as~$s \to \exp(- \hvarphi(x_0))^-$, there is~$s_0$ in~$\left( 0, \exp( - \hvarphi(x_0)) \right)$ such that~$\Phi(s_0) \ge 1$.
It follows that the radius of convergence of~$\Upsilon$, and hence that of~$\Xi$, is less than or equal to~$s_0$ and therefore it is strictly less than~$\exp(-
\hvarphi(x_0))$.
This completes the proof of~\eqref{e:induced gap pressure} and of the lemma.
\end{proof}

\begin{proof}[Proof of the Key Lemma]
When~$\nu$ is supported on a periodic orbit, the desired inequality follows from Lemma~\ref{l:periodic case} in the real case.
In the complex case, note first that the hypotheses that~$\beta > 1$ and that~$f$ satisfies the Polynomial Shrinking of Components condition with exponent~$\beta$, imply that~$f$ has no neutral periodic point in~$J(f)$.
So, in this case the Key Lemma follows from Case~$1$ in the proof of~\cite[Proposition~$4.1$]{InoRiv12}.

Suppose~$\nu$ is not supported on a periodic orbit.
By Proposition~\ref{p:constructing a free IMFS} there is~$D > 0$, a connected and compact subset~$B_0$ of~$\dom(f)$ that intersects~$J(f)$, and a free IMFS~$(\phi_k)_{k = 1}^{+ \infty}$ generated by~$f$ with time sequence~$(m_k)_{k= 1}^{+ \infty}$ that is defined on~$B_0$, and such that for every~$k\ge 1$ and every point~$y$ in~$\phi_k(B_0)$ we have
\begin{equation}
\label{e:optimal sum}
S_{m_k}(\varphi)(y)
\ge
m_k\int_{J(f)} \varphi \ d\nu -D.
\end{equation}
Since the IMFS~$(\phi_k)_{k = 1}^{+ \infty}$ is free, there is a point~$x_0$ of~$B_0$ in~$J(f)$ such that for every~$\ul$ and~$\ul'$ in~$\Sigma^*$ such that~$m_{\ul} = m_{\ul'}$, the sets~$\phi_{\ul}(x_0)$ and~$\phi_{\ul'}(x_0)$ are disjoint.
Note that for every integer~$k \ge 1$, every~$\ul = l_1 \cdots l_k$ in~$\Sigma^*$, every~$y_0$ in~$\phi_{\ul}(x_0)$, and every~$j$ in~$\{1, \ldots, k-1 \}$, the point
$$ y_j \= f^{m_{l_1}+m_{l_2}+\cdots +m_{l_{j}}}(y_0) $$
is in~$\phi_{m_{l_{j+1}}}(B_0)$.
Therefore, by~\eqref{e:optimal sum} we have
\begin{multline*}
S_{m_{\ul}}(\varphi)(y_0)
=
S_{m_{l_1}}(\varphi)(y_0) + S_{m_{l_2}}(\varphi)(y_1) + \cdots + S_{m_{l_k}}(\varphi)(y_{k - 1})
\\ \ge
\sum_{i=1}^k \left( m_{l_i}\int_{J(f)} \varphi \ d\nu - D \right)
=
m_{\ul}\int_{J(f)} \varphi \ d\nu - kD.
\end{multline*}
This shows that for every~$\ul$ in~$\Sigma^*$, and every~$y_0$ in~$\phi_{\ul}(x_0)$ we have
\begin{equation}\label{e:bigsum}
\exp (S_{m_{\ul}}(\varphi)(y_0))\ge \exp
\left(m_{\ul}\int_{J(f)} \varphi \ d\nu\right)
\exp(-|\ul|D).
\end{equation}
On the other hand, if for every integer~$n \ge 1$  we put
$$ \Xi_n
\=
\bigcup_{\ul \in \Sigma^*, m_{\ul} = n} \phi_{\ul}(x_0), $$
then the radius of convergence of the series
$$ \Xi(s)
\=
\sum_{n= 1}^{+ \infty}\left(\sum_{y\in \Xi_n}\exp (S_n(\varphi)(y))\right)s^n, $$
is given by
$$ R \= \left(\limsup_{n\to + \infty} \left(\sum_{y\in \Xi_n}\exp (S_n(\varphi)(y))\right)^{1/n}\right)^{-1}, $$
and satisfies
$$\exp \left(-\limsup_{n\rightarrow + \infty}\frac{1}{n}\log \sum_{y\in f^{-n}(x_0)}\exp (S_n(\varphi)(y))\right)\le R. $$
Therefore, to complete the proof of the Key Lemma it suffices to prove~$R < \exp \left( - \int_{J(f)} \varphi \ d\nu \right)$.
Put
$$ \Phi(s)
\=
\sum_{l=1}^{+ \infty} \exp(-D)\exp \left(m_{l}\int_{J(f)} \varphi \ d\nu\right) s^{m_l}. $$
By inequality~\eqref{e:bigsum} and the fact that $(\phi_k)_{k = 1}^{+ \infty}$ is free, each of the coefficients of the series
\begin{equation*}
\Upsilon(s)
\=
\sum_{i=1}^{+ \infty}\Phi(s)^i
=
\sum_{n = 1}^{+ \infty} \left(\sum_{\ul \in \Sigma^*, m_{\ul} = n}\exp \left(m_{\ul}\int_{J(f)} \varphi \ d\nu\right) \exp(-|\ul| D)\right)s^n,
\end{equation*}
does not exceed the corresponding coefficient of the series~$\Xi$, so the radius of convergence of~$\Xi$ is less than or equal to that of~$\Upsilon$.
Since clearly~$\Phi(s) \to + \infty$ as~$s \to \exp\left(-\int_{J(f)} \varphi \ d\nu\right)^-$, there exists~$s_0$ in~$\left( 0, \exp\left(-\int_{J(f)} \varphi \ d\nu\right) \right)$ such that~$\Phi(s_0)\ge 1$.
This implies that the radius of convergence of~$\Upsilon$, and hence that of~$\Xi$, is less than or equal to~$s_0$, and therefore that~$R \le s_0 < \exp \left(-\int_{J(f)} \varphi \ d\nu \right)$.
This completes the proof of the Key Lemma.
\end{proof}

% For Bibtex bibliography:
\bibliographystyle{alpha}
% \bibliography{$HOME/papers/0BIB/papers}

\end{document}